\documentclass[11pt]{article}
\usepackage[english]{babel}
\usepackage{amssymb,amsmath,amsthm}
\usepackage{MnSymbol}
\usepackage{graphicx}
\newtheorem{theorem}{Theorem}[section]
\newtheorem{lemma}[theorem]{Lemma}
\newtheorem{proposition}[theorem]{Proposition}
\newtheorem{corollary}[theorem]{Corollary}

{\theoremstyle{definition}}
{\theoremstyle{definition}}
{\theoremstyle{definition}\newtheorem{remark}[theorem]{Remark}}

\numberwithin{equation}{section}

\def\C{{\mathbb C}}

\def\Z{{\mathbb Z}}

\def\K{{\mathbb K}}
\def\epsilon{\varepsilon}
\def\kappa{\varkappa}
\def\phi{\varphi}
\def\leq{\leqslant}
\def\geq{\geqslant}
\def\dim{{\rm dim}\,}

\title{Sklyanin algebras and a cubic root of 1}

\author{Natalia Iyudu and Stanislav Shkarin}

\date{}

\begin{document}

\maketitle

\begin{abstract}We consider Sklyanin algebras $S$ with 3 generators, which are quadratic algebras over a field $\K$ with $3$ generators $x,y,z$ given by $3$ relations $pxy+qyx+rzz=0$, $pyz+qzy+rxx=0$ and $pzx+qxz+ryy=0$, where $p,q,r\in\K$. This class of algebras has enjoyed much attention. In particular, using tools from algebraic geometry Artin, Tate and Van Den Berg \cite{ATV2}
  showed that if at least two of the parameters $p$, $q$ and $r$  are non-zero and at least two of three numbers $p^3$, $q^3$ and $r^3$ are distinct, then $S$ is Artin--Schelter regular.  More specifically, $S$ is Koszul and has the same Hilbert series as the algebra of commutative polynomials in 3 indeterminates. It has became commonly accepted that it is impossible to achieve the same objective by purely algebraic and combinatorial means like the Gr\"obner basis technique. The authors have previously dispelled this belief. However our previous proof was no less complicated than the one based on algebraic geometry. It used a construcion of a Gr\"obner basis in a suitable one-sided module over $S$ and had quite a number of cases to consider. In this paper we exhibit a linear substitution after which it becomes possible to determine the leading monomials of a reduced Gr\"obner basis for the ideal of relations of $S$ itself (without passing to a module). We also find out explicitly (in terms of parameters) which Sklyanin algebras are isomorphic. The only drawback of the new technique is that it fails if the characteristic of the ground field equals 3.
\end{abstract}



\section{Introduction \label{s1}}\rm

Throughout this paper $\K$ is an arbitrary field of characteristic different from $3$. If $B$ is a graded algebra, the symbol $B_m$ stands for the $m^{\rm th}$ graded component of the algebra $B$. If $V$ is an $n$-dimensional vector space over $\K$, then $F=F(V)$ is the tensor algebra of $V$. For any choice of a basis $x_1,\dots,x_n$ in $V$, $F$ is naturally identified with the free $\K$-algebra with the generators $x_1,\dots,x_n$. For subsets $P_1,\dots,P_k$ of an algebra $B$, $P_1\dots P_k$ stands for the linear span of all products $p_1\dots p_k$ with $p_j\in P_j$. We consider a degree grading on the free algebra $F$: the $m^{\rm th}$ graded component of $F$ is $V^m$. If $R$ is a subspace of the $n^2$-dimensional space $V \otimes V$, then the quotient of $F$ by the ideal $I$ generated by $R$ is called a {\it quadratic algebra} and denoted $A(V,R)$. For any choice of bases $x_1,\dots,x_n$ in $V$ and $g_1,\dots,g_k$ in $R$, $A(V,R)$ is the algebra given by generators $x_1,\dots,x_n$ and the relations $g_1,\dots,g_k$ ($g_j$ are linear combinations of monomials $x_ix_s$ for $1\leq i,s\leq n$). Since each quadratic algebra $A$ is degree graded, we can consider its Hilbert series
$$
H_A(t)=\sum_{j=0}^\infty {\rm dim}_{\K} A_j\,\,t^j.
$$

Quadratic algebras whose Hilbert series is the same as for the algebra $\K[x_1,\dots,x_n]$ of commutative polynomials play a particularly important role in physics. We say that $A$ is a {\it PHS-algebra} (for {\it polynomial Hilbert series}) if
$$
H_A(t)=H_{\K[x_1,\dots,x_n]}(t)=(1-t)^{-n}.
$$
Following the notation from the Polishchuk, Positselski book \cite{popo}, we say that a quadratic algebra $A=A(V,R)$ is a {\it PBW-algebra} (Poincare, Birkhoff, Witt) if there are linear bases $x_1,\dots,x_n$ and $g_1,\dots,g_m$ in $V$ and $R$ respectively such that with respect to some compatible with multiplication well-ordering on the monomials in $x_1,\dots,x_n$, $g_1,\dots,g_m$ is a (non-commutative) Gr\"obner basis of the ideal $I_A$ generated by $R$. In this case, $x_1,\dots,x_n$ is called a {\it PBW-basis} of $A$, while $g_1,\dots,g_m$ are called the {\it PBW-generators} of $I_A$. In order to avoid confusion, we would like to stress from the start that Odesskii \cite{ode} as well as some other authors use the term PBW-algebra for what we have already dubbed PHS. Since we deal with both concepts, we could not possibly call them the same and we opted to follow the notation from \cite{popo}.

Another concept playing an important role in this paper is Koszulity. For a quadratic algebra $A=A(V,R)$, the augmentation map $A\to \K$ equips $\K$ with the structure of a commutative graded $A$-bimodule. The algebra $A$ is called {\it Koszul} if $\K$ as a graded right $A$-module has a free resolution $\dots\to M_m\to\dots\to M_1\to A\to\K\to 0$ with the second last arrow being the augmentation map and with each $M_m$ generated in degree $m$. The last property is the same as the condition that the matrices of the above maps $M_m\to M_{m-1}$ with respect to some free bases consist of elements of $V$ (=are homogeneous of degree $1$).

The notion of a {\it noncommutative potential} was  introduced  in \cite{Ko}. We make use of an equivalent definition from \cite{BW}. An element $F$ of $\K\langle x_1,\dots,x_n\rangle$ is called {\it cyclicly invariant} if $F$ is invariant for the linear map $C:\K\langle x_1,\dots,x_n\rangle\to \K\langle x_1,\dots,x_n\rangle$ defined by its action on monomials as follows: $C(1)=1$ and $C(x_ju)=ux_j$ for every $j$ and every monomial $u$. The symbol $\K^{\rm cyc}\langle x_1,\dots,x_n\rangle$ stands for the vector space of all cyclicly invariant elements of $\K\langle x_1,\dots,x_n\rangle$. We also consider the linear maps $\frac{\delta}{\delta x_j}:\K\langle x_1,\dots,x_n\rangle\to \K\langle x_1,\dots,x_n\rangle$ defined by their action on monomials $u$ as follows: $\frac{\delta u}{\delta x_j}=0$ if $u$ does not start with $x_j$ and $\frac{\delta u}{\delta x_j}=v$ if $u=x_jv$. A potential algebra $A_F$ defined by a potential $F\in \K^{\rm cyc}\langle x_1,\dots,x_n\rangle$ is a $\K$-algebra given by the generators $x_1,\dots, x_n$ and the relations $\frac{\delta F}{\delta x_j}=0$ for $1\leq j\leq n$. For the sake of convenience, we consider the onto linear map $G\mapsto G^\rcirclearrowleft$ from $\K\langle x_1,\dots,x_n\rangle$ onto $\K^{\rm cyc}\langle x_1,\dots,x_n\rangle$ defined by its action on homogeneous elements by $u^\rcirclearrowleft=C(u)+{\dots}+C^{d}u$, where $d$ is the degree of $u$. Foe example, ${x^4}^\rcirclearrowleft=4x^4$ and ${x^2y}^\rcirclearrowleft=x^2y+xyx+yx^2$.

Prime examples of a potential algebras are Sklyanin algebras. Recall that if $(p,q,r)\in \K^3$, the {\it Sklyanin algebra} $Q^{p,q,r}$ is the quadratic algebra over $\K$ with generators $x,y,z$ given by $3$ relations
\begin{equation*}
\text{$pyz+qzy+rxx=0$},\quad \text{$pzx+qxz+ryy=0$},\quad\text{$pxy+qyx+rzz=0$}.
\end{equation*}
It is easy to observe that $Q^{p,q,r}$ is a potential algebra given by the potential $r(x^3+y^3+z^3)+pxyz^\rcirclearrowleft+qxzy^\rcirclearrowleft$.

It is worth mentioning that our use of potentiality of Sklyanin algebras is just a matter of convenience. Namely, it makes the job of parforming a linear substitution much easier. Instead of making the sub in each of the three relations, one can perform it with the potential and then compute the 'derivatives' $\frac{\delta}{\delta x_j}$ yielding the same result.

Odesskii \cite{ode} proved that in the case $\K=\C$, a generic Sklyanin algebra is a PHS-algebra. That is,
\begin{equation*}
\textstyle H_{Q^{p,q,r}}(t)=\sum\limits_{j=0}^\infty \frac{(j+2)(j+1)}{2}\,t^j\ \ \text{for generic $(p,q,r)\in\C^3$,}
\end{equation*}
where generic means outside the union of countably many  algebraic varieties in $\C^3$ (different from $\C^3$). In particular, the equality  above holds for almost all $(p,q,r)\in\C^3$ with respect to the Lebesgue measure. Polishchuk and Positselski \cite{popo} showed in the same setting and with the same meaning of the word 'generic', that for generic $(p,q,r)\in\C^2$, the algebra $Q^{p,q,r}$ is Koszul but is not a PBW-algebra. The same results are implicitly contained in the Artin, Shelter paper \cite{AS}.

Artin, Tate and Van den Berg \cite{ATV2,ATV1}, and Feigin, Odesskii \cite{odf}, considered certain family of infinite dimensional representations of Sklyanin algebra. Namely, they used representations, where variables are represented by infinite matrices with one nonzero upper diagonal. Equivalently, they considered graded modules with all graded components being  one-dimensional. The geometric interpretation of the space of such modules is in the core for most of their arguments. Artin, Tate and Van den Berg showed that if at least two of the parameters $p$, $q$ and $r$  are non-zero and the equality $p^3=q^3=r^3$ fails, then $Q^{p,q,r}$ is Artin--Shelter regular. More specifically, $Q^{p,q,r}$ is Koszul and has the same Hilbert series as the algebra of commutative polynomials in three variables.

It became commonly accepted that it is impossible to obtain the same results by purely algebraic and combinatorial means like the Gr\"obner basis technique, see, for instance, comments in \cite{ode,wen}. We have dispelled this notion in \cite{our}. However our previous proof is rather complicated. It uses a construcion of a Gr\"obner basis in a suitable one-sided module over $Q^{p,q,r}$ and has quite a number of cases to consider. In this paper we exhibit a linear substitution after which it becomes possible to determine the leading monomials of a reduced Gr\"obner basis for the ideal of relations of $S$ itself (without passing to a module). We also find out explicitly (in terms of parameters) which Sklyanin algebras are isomorphic. The only drawback of the new technique is that it fails if the characteristic of the ground field equals 3.

\begin{theorem}\label{copo0} The algebra $Q^{p,q,r}$ is Koszul for any $(p,q,r)\in\K^3$. The algebra $Q^{p,q,r}$ is PHS
if and only if at least two of $p$, $q$ and $r$ are non-zero and the equality $p^3=q^3=r^3$ fails.
\end{theorem}

We stress again that the above theorem is essentially one of the main results in \cite{ATV2} with a different proof provided in \cite{our}. Our new proof here is however much easier and shorter, since we found  clever change of variables  which makes the Gr\"obner basis  regular (which is a generalisation of finiteness).

\section{General background \label{s2}}

We shall use the following well-known facts, all of which can be found in \cite{popo}. Every monomial quadratic algebra $A=A(V,R)$ (=there are linear bases $x_1,\dots,x_n$  and $g_1,\dots,g_m$ in $V$ and $R$ respectively, such that each $g_j$ is a monomial in $x_1,\dots,x_n$) is a PBW-algebra. Next, if we pick a basis $x_1,\dots,x_n$ in $V$, we get a bilinear form $b$ on the free algebra $F=F(V)$ defined by $b(u,v)=\delta_{u,v}$ for every monomials $u$ and $v$ in the variables $x_1,\dots,x_n$. The algebra $A^!=A(V,R^\perp)$, where $R^\perp=\{u\in V^2:b(r,u)=0\ \text{for each}\ r\in R\}$, is called the {\it dual algebra} of $A$. Clearly, $A^!$ is a quadratic algebra in its own right. Recall also that there is a specific complex of free right $A$-modules, called the Koszul complex, whose exactness is equivalent to the Koszulity of $A$:
\begin{equation}\label{koco1}
\dots\mathop{\longrightarrow}^{d_{k+1}} (A^!_k)^*\otimes A\mathop{\longrightarrow}^{d_k} (A^!_{k-1})^*\otimes A
 \mathop{\longrightarrow}^{d_{k-1}}\dots \mathop{\longrightarrow}^{d_1} (A^!_{0})^*\otimes A=A\longrightarrow \K\to 0,
\end{equation}
where the tensor products are over $\K$, the second last arrow is the augmentation map, each tensor product carries the natural structure of a free right $A$-module and $d_k$ are given by $d_k(\phi \otimes u)=\sum\limits_{j=1}^n \phi_j\otimes x_ju$, where $\phi_j\in (A^!_{k-1})^*$, $\phi_j(v)=\phi(x_jv)$. Although $A^!$ and the Koszul complex seem to depend on the choice of a basis in $V$, it is not really the case up to the natural equivalence \cite{popo}. We recall that
\begin{align}\notag
&\text{every PBW-algebra is Koszul;}
\\
\notag
&\text{$A$ is Koszul $\iff$ $A^!$ is Koszul};
\\
\label{stm2}
&\text{if $A$ is Koszul, then $H_A(-t)H_{A^!}(t)=1$}.
\end{align}

Note that the Koszul complex (\ref{koco1}) of any quadratic algebra is exact at its last $3$ terms: $\K$, $(A^!_0)^*\otimes A=A$ and $(A^!_{1})^*\otimes A$. The following lemma proved in \cite{our} allows us to prove Koszulity of Sklyain algebras once we have computed their Hilbert series. We say that $u\in A=A(V,R)$ is a {\it right annihilator} if $Vu=\{0\}$ in $A$. A right annihilator $u$ is {\it non-trivial} if $u\neq 0$.

\begin{lemma}\label{deg3a} Let $A=A(V,R)$ be a quadratic algebra such that $A^!_4=\{0\}$, $A^!_{3}$ is one-dimensional and $wA_2^!\neq\{0\}$ for every non-zero $w\in A_1^!$. Then the following statements are equivalent$:$
\begin{itemize}\itemsep=-2pt
\item[\rm (\ref{deg3a}.1)] $A$ is Koszul$;$
\item[\rm (\ref{deg3a}.2)] $A$ has no non-trivial right annihilators and $H_A(-t)H_{A^!}(t)=1$.
\end{itemize}
\end{lemma}

Our next observation is that neither Koszulity nor the Hilbert series of a quadratic algebra $A=A(V,R)$ is sensitive to changing the ground field.

\begin{remark}\label{rem1}Fix the bases $x_1,\dots,x_n$ and $r_1,\dots,r_m$ in $V$ and $R$ respectively. Then $A=A(V,R)$ is given by the generators $x_1,\dots,x_n$ and the relations $r_1,\dots,r_m$. Let $\K_0$ be the subfield of $\K$ generated by the coefficients in the relations $r_1,\dots,r_m$ and $B$ be the $\K_0$-algebra defined by the exact same generators $x_1,\dots,x_n$ and the exact same relations $r_1,\dots,r_m$. Then $A$ is Koszul if and only if $B$ is Koszul (see, for instance, \cite{popo}) and the Hilbert series of $A$ and of $B$ coincide. The latter follows from the fact that the Hilbert series depends only on the set of leading monomials of the Gr\"obner basis. Now the Gr\"obner basis construction algorithm for $A$ and for $B$ produces exactly the same result. Thus if a quadratic algebra given by generators and relations makes sense over 2 fields of the same characteristic, then the choice of the field does not effect its Hilbert series or its Koszulity. In particular, replacing the original field $\K$ by its algebraic closure or by an even bigger field does not change the Hilbert series or Koszulity of $A$. On the other hand, the PBW-property is sensitive to changing the ground field \cite{popo}.
\end{remark}

\section{Elementary observations \label{observ}}

The following elementary facts are proved in \cite{our}. The next lemma is proved by a direct computation of the reduced Gr\"obner basis for the ideal of relations of the duals of Sklyanin algebras (using the usual left-to-right degree lexicographical ordering on monomials assuming $x>y>z$. The last statement is verified directly since knowing the Gr\"oner basis we know the multiplication table in  the relevant finite dimensional (in this case) algebra $A^!$.

\begin{lemma} \label{duser} Let $(p,q,r)\in\K^3$ and $A=Q^{p,q,r}$. Then the Hilbert series of $A^!$ is given by
\begin{equation}\label{hsaa}
H_{A^!}(t)=\left\{\begin{array}{ll}1+3t&\text{if $p=q=r=0$;}
\\ \textstyle\frac{1+2t}{1-t}&\text{if $p^3=q^3=r^3\neq 0$ or exactly two of $p$, $q$ and $r$ equal $0$;}\\ (1+t)^{3}&\text{otherwise}.
\end{array}\right.
\end{equation}
Moreover, $A_2^!w\neq\{0\}$ for each non-zero $w\in A_1^!$ provided $H_{A^!}(t)=(1+t)^{3}$.
\end{lemma}

Note \cite{popo} that for every quadratic algebra $A=A(V,R)$ (Koszul or otherwise), the power series $H_A(t)H_{A^!}(-t)-1$ starts with $t^k$ with $k\geq 4$. This allows to determine $\dim A_3$ provided we know $\dim A^!_j$ for $j\leq 3$. Applying this observation together with (\ref{hsaa}), we immediately obtain the following fact.

\begin{corollary} \label{a3} Let $(p,q,r)\in\K^3$ and $A=Q^{p,q,r}$. Then
\begin{equation}\label{hsaa1}
\dim A_3=\left\{\begin{array}{ll}27&\text{if $p=q=r=0;$}
\\ 12&\text{if $p^3=q^3=r^3\neq 0$ or exactly two of $p$, $q$ and $r$ equal $0;$}\\ 10&\text{otherwise}.
\end{array}\right.
\end{equation}
\end{corollary}

Obviously, multiplying $(p,q,r)\in\K^3$ by a non-zero scalar does not change the algebra $Q^{p,q,r}$. It turns out that there are non-proportional triples of parameters, which lead to isomorphic (as graded algebras) Sklyanin algebras. The following three lemmas can be found in \cite{our}. Since their proof is very short, we present it for the sake of reader's convenience.

\begin{lemma}\label{root1} Assume that $(p,q,r)\in\K^3$ and $\theta\in\K$ is such that $\theta^3=1$ and $\theta\neq 1$. Then the graded algebras $Q^{p,q,r}$ and $Q^{p,q,\theta r}$ are isomorphic.
\end{lemma}

\begin{proof} The relations of $Q^{p,q,r}$ in the variables $u$, $v$, $w$ given by $x=u$, $y=v$ and $z=\theta^2w$ read
$puv+qvu+\theta rww=0$, $pwu+quw+\theta rvv=0$ and $pvw+qwv+\theta ruu=0$. Thus this change of variables provides an isomorphism between $Q^{p,q,r}$ and $Q^{p,q,\theta r}$.
\end{proof}

\begin{lemma}\label{root2} Assume that $(p,q,r)\in\K^3$ and $\theta\in\K$ is such that $\theta^3=1$ and $\theta\neq 1$. Then the graded algebras $Q^{p,q,r}$ and $Q^{p',q',r'}$ are isomorphic, where $p'=\theta^2p+\theta q+r$, $q'=\theta p+\theta^2q+r$ and $r'=p+q+r$.
\end{lemma}

\begin{proof} A direct computation shows that the space of the quadratic relations of $Q^{p,q,r}$ in the variables $u$, $v$, $w$ given by $x=u+v+w$, $y=u+\theta v+\theta^2 w$ and $z=u+\theta^2v+\theta w$ (the matrix of this change of variables is non-degenerate) is spanned by $p'uv+q'vu+r'ww=0$, $p'wu+q'uw+r'vv=0$ and $p'vw+q'wv+r'uu=0$. Thus $Q^{p,q,r}$ and $Q^{p',q',r'}$ are isomorphic.
\end{proof}

\begin{lemma}\label{dege1} The Sklyanin algebra $A=Q^{p,q,r}$ is PBW and therefore is Koszul if $r=0$, or if $p=q=0$, or if $p^3=q^3=r^3$. Moreover, $H_A(t)=(1-3t)^{-1}$ if $p=q=r=0$, $H_A(t)=\frac{1+t}{1-2t}$ if exactly two of $p$, $q$ and $r$ are $0$ or if $p^3=q^3=r^3\neq 0$ and $H_A=(1-t)^{-3}$ if $r=0$ and $pq\neq 0$. Furthermore, $A$ is Koszul and satisfies $H_A=(1-t)^{-3}$ if $(p+q)^3=-r^3\neq 0$ and $p^3=q^3=r^3$ fails.
\end{lemma}

\begin{proof} First, if $p=q=r=0$, then $Q^{p,q,r}$ is the free algebra and therefore $A=Q^{p,q,r}$ is PBW and therefore Koszul and has the Hilbert series $H_A(t)=(1-3t)^{-1}$. If exactly two of $p$, $q$ and $r$ are $0$, then $A$ is monomial and therefore is PBW and therefore Koszul. In this case $H_A(t)=\frac{1+t}{1-2t}$. If $p^3=q^3=r^3\neq 0$, one easily checks that the defining relations of $A$ form a Gr\"obner basis in the ideal they generate. Hence $A$ is PBW and therefore Koszul. Furthermore, the Hilbert series of $A$ is the same as for the monomial algebra given by the leading monomials $xx$, $xy$ and $xz$  of the relations of $A$. It follows that again $H_A(t)=\frac{1+t}{1-2t}$. If $r=0$ and $pq\neq 0$, the defining relations form a Gr\"obner basis again. This time the leading monomials are $xy$, $xz$ and $yz$, the same as for the algebra $\K[x,y,z]$ of commutative polynomials. Hence $A$ is a PBW (and therefore Koszul) PHS-algebra. The latter means that $H_A=(1-t)^{-3}$. As a matter of fact, $A$ in this case is an algebra of quantum polynomials on three variables. Finally, assume that $(p+q)^3=-r^3\neq 0$ and $p^3=q^3=r^3$ fails. By Remark~\ref{rem1}, by passing to a field extension, we can, without loss of generality, assume that $\K$ is algebraically closed. Since $(p+q)^3=-r^3\neq 0$, Lemma~\ref{root1} implies that $A=Q^{p,q,r}$ is isomorphic to $Q^{p,q,-p-q}$. Thus without loss of generality $r=-p-q$. Since ${\rm char}\,\K\neq 3$ and $\K$ is algebraically closed, we can find $\theta\in\K$ such that $\theta^3=1\neq\theta$. By Lemma~\ref{root2}, $A$ is isomorphic to $Q^{p',q',r'}$, where $p'=\theta^2p+\theta q+r=(\theta^2-1)p+(\theta-1)q$, $q'=\theta p+\theta^2q+r=(\theta-1)p+(\theta^2-1)q$ and $r'=p+q+r=0$. If $p'q'\neq 0$, we are back to quantum polynomials and the result follows. Observe that $p'+q'=-3(p+q)$. Thus the case $p'+q'=0$ is impossible since otherwise $p+q=0$, which conradicts $(p+q)^3=-r^3\neq 0$. In particular, it is impossible to have $p'=q'=0$. Thus it remains to consider the case when exactly one of $p'$ or $q'$ is zero. By Corollary~\ref{a3}, we then have $\dim A_3=\dim Q^{p',q',r'}=12$. Since condition $(p+q)^3=-r^3\neq 0$ implies that at least two of $p,q,r$ are non-zero, the same Corollary~\ref{a3} yields $p^3=q^3=r^3\neq 0$, which contradicts the assumptions.
\end{proof}

\begin{lemma}\label{dege2} The Sklyanin algebra $A=Q^{p,q,r}$ is Koszul and satisfies $H_A=(1-t)^{-3}$ if $p^3=q^3\neq 0$.
\end{lemma}

\begin{proof} If $r=0$, the result follows from Lemma~\ref{dege1}. Thus we can assume that $r\neq 0$. By normalizing, we can without loss of generality assume $p=1$. Then $q^3=1$ and the defining relations take form $xx=-\frac1r yz-\frac qrzy$, $xy=-qyx-rzz$ and $xz=-\frac1q zx-\frac rq yy$. A direct computation shows that the reduced Gr\"ober basis of the ideal of relations comprises the defining relations $rxx+yz+qzy$, $xy+qyx+rzz$, $qxz+ryy+zx$ together with $yyz-q^2 zyy$ and $yzz-q^2zzy$ (the basis is finite; one has to use the equality $q^3=1$). Now the normal words (the words, which do not contain any of the leading monomials $xx$, $xy$, $xz$, $yyz$, $yzz$ of the basis as submonomials) are exactly $z^k(yz)^my^lx^\epsilon$, where $k,m,l$ are non-negative integers and $\epsilon\in\{0,1\}$. As there are precisely $\frac{(n+1)(n+2)}{2}$ normal words of degree $n$, we have $H_A=(1-t)^{-3}$. Since the set of normal words is closed under multiplication by $z$ on the left, the map $u\mapsto zu$ from $A$ to itself is injective. In particular, $A$ has no non-trivial right annihilators. By Lemma~\ref{duser} $H_{A^!}=(1+t)^3$ and $wA_2^!\neq \{0\}$ for every non-zero $w\in A_1^!$. Now Lemma~\ref{a3} implies that $A$ is Koszul.
\end{proof}

We need another observation made in \cite{our}.

\begin{lemma} \label{loest}
For every $p,q,r\in\K$, the Sklyanin algebra $A=Q^{p,q,r}$ satisfies $\dim A_n\geq \frac{(n+1)(n+2)}{2}$ for every $n\in\Z_+$.
\end{lemma}

\section{Proof of Theorem~\ref{copo0} \label{s3}}

Throughout this section $p,q,r\in \K$ and $A=Q^{p,q,r}$. By Remark~\ref{rem1}, by passing to a field extension, we can, without loss of generality, assume that $\K$ is algebraically closed. Since ${\rm char}\,\K\neq 3$ and $\K$ is algebraically closed, we can find $\theta\in\K$ such that $\theta^3=1\neq\theta$. If $p^3=q^3=r^3$ or $p=q=0$ or $r=0$ or $(p+q)^3+r^3=0$ or $p^3=q^3$, the conclusion of Theorem~\ref{copo0} follows from Lemma~\ref{dege1}. Thus for the rest of the proof, we can assume that none of these equalities holds. Moreover, if $p+q=0$, one easily sees that a substitution from Lemma~\ref{root2} to breaks this equality. Thus we can additionally assume that $p+q\neq 0$.

Since $r\neq 0$, by scaling the relations, we can without loss of generality assume that $r=1$.  Then $(p,q)\neq (0,0)$, $(p^3-1,q^3-1)\neq (0,0)$, $p+q\neq 0$ and $(p+q)^3+1\neq 0$.

Our proof hinges on finding a convenient linear substitution. The main objective is to make specific monomials (namely, $xx$, $xy$ and $yz$) into leading monomials of defining relations with respect to the standard left-to-right degree-lexicographical ordering assuming $x>y>z$. We perform the substitution in a number of steps: each of the steps is a linear sub itself and the resultipng substitution is their composition. We keep the same letters $x,y,z$ for both old and new variables. We introduce a substitution by showing by which linear conmbination of (new) $x,y,z$ must the (old) variables be replaced. For example, if we write $x\to x+y+z$, $y\to z-y$ and $z\to 7z$, this means that all occurances of $x$ (in the relations, potential etc.) are replaced by $x+y+z$, all occurances of $y$ are replaced by $z-y$, while $z$ is swapped for $7z$.

Note that our Sklyanin algebra $Q^{p,q,1}$ is potential with the potential
$$
F_{p,q}=x^3+y^3+z^3+pxyz^\rcirclearrowleft+qxzy^\rcirclearrowleft.
$$
First, we perform the sub $x\to -\frac{x}{p+q}$, $y\to y$ and $z\to z$. As a result, we see that $A=Q^{p,q,1}$ is isomorphic to the potential algebra with the potential
$$
\textstyle F'_{p,q}=-\frac{(p+q)^3+1}{(p+q)^3}x^3+(x^3+y^3+z^3)-\frac{p}{p+q}xyz^\rcirclearrowleft-\frac{q}{p+q}xzy^\rcirclearrowleft.
$$
Note that we have used the condition $p+q\neq 0$. Next, we do the sub $x\to x+y+z$, $y\to x+\theta^2y+\theta z$ and $z\to x+\theta y+\theta^z$. As a result, we see that $A=Q^{p,q,1}$ is isomorphic to the potential algebra with the potential
$$
\textstyle F''_{p,q}=-\frac{(p+q)^3+1}{(p+q)^3}(x+y+z)^3+\frac{3((1-\theta)p+(1-\theta^2)q)}{p+q}xyz^\rcirclearrowleft
+\frac{3((1-\theta^2)p+(1-\theta)q)}{p+q}xzy^\rcirclearrowleft.
$$
Note that the $(x+y+z)^3$ coefficient in $F''_{p,q}$ is non-zero since $(p+q)^3+1\neq 0$. Since scaling multiplying by a potential by a non-zero constant has no effect on the corresponding potential algebra, $A$ is isomorphic to the potential algebra with the potential
$$
G_{a,b}=(x+y+z)^3+axyz^\rcirclearrowleft+bxzy^\rcirclearrowleft,
$$
where
$$\textstyle
a=\frac{3(p+q)^2((\theta-1)p+(\theta^2-1)q)}{(p+q)^3+1},\quad b=\frac{3(p+q)^2((\theta^2-1)p+(\theta-1)q)}{(p+q)^3+1}.
$$
Note that $a=0\iff p=\theta^2q$, $b=0\iff p=\theta q$ and $a=b\iff p=q$. Since $p^3\neq q^3$, we have $ab(a-b)\neq 0$. Furthermore, $a+b=-\frac{9(p+q)^3}{(p+q)^3+1}$ and therefore $a+b\neq 0$ as well.

Now we perform the sub $x\to \frac{a}{a+b}x$, $y\to \frac{b}{a+b}x+y-z$ and $z\to z$. Note that we need the fact that $ab(a+b)\neq 0$ for this sub to be non-degenerate. As a result, we see that $A=Q^{p,q,1}$ is isomorphic to the potential algebra with the potential
$$
G'_{a,b}=(x+y)^3+\frac{ab}{a+b}xxz^\rcirclearrowleft+\frac{a^2}{a+b}xyz^\rcirclearrowleft+\frac{ab}{a+b}xzy^\rcirclearrowleft-axzz^\rcirclearrowleft.
$$
By this point we have already reached our main objective. One easily sees that the leading monomials of the defining relations of the potential algebra with the potential $G'_{a,b}$ are indeed $xx$, $xy$ and $yz$. However, we shall perform one final sub (with a triangular matrix) in order to simplify and/or kill some of the coefficients. Namely, we use the sub $x\to x-\frac{a-b}{a}y+\frac{(a+b)^2+a^2b}{(a-b)^3}z$, $y\to \frac{a-b}{a}y-\frac{(a+b)^2+ab^2}{(a-b)^3}z$ and $z\to -\frac{a+b}{(a-b)^2}z$. Note that the sub is non-degenerate since $ab(a+b)(a-b)\neq 0$. As a result, we see that $A$ is isomorphic to the potential algebra with the potential
$$
P_{\alpha,\gamma}=x^3-xyz^\rcirclearrowleft+\alpha yyz^\rcirclearrowleft-(\gamma-\alpha^2)z^3,
$$
where
\begin{align*}
\textstyle\alpha&=-\frac{(a+b)^3+ab(a^2+b^2)}{(a-b)^4}
\\
\textstyle\gamma&=-\frac{(a+b)^4(a^2-ab+b^2)+ab(a+b)^3(2a^2+2b^2-3ab)+a^2b^2(a^4+b^4+a^2b^2-a^3b-ab^3)}{(a-b)^8}
\end{align*}

\begin{remark}\label{rem2}
There is no black magic to this string of substitutions. We have tried this because there is no visible pattern to the leading monomials of the Gr\"obner basis of the ideal of relations of the Sklyanin algebras in their original form. We had to change something and the most drastic change ensues from altering the leading monomials of the defining relations. However for generic Sklyanin, it is impossible to get rid of the square $xx$ of the largest generator. It is equally impossible to lose both $xy$ and $xz$, so we might as well keep $xy$. The only freedom we have now is to find a sub, which will change the smallest leading term $xz$ and this can not possibly go below $yz$. So we set on changing it into $yz$. Next, we used elementary linear algebra to find necessary and sufficient conditions on a potential for the corresponding potential algebra to have $xy$, $xz$ and $yz$ as the leading terms of defining relations. These have the form of equations on the coefficients of the potential. Each of the first few of the above subs were designed to alter the potential in such a way that one equation is satisfied without spoiling the previously obtained ones. The last sub is there for the sake of neatness.
\end{remark}

The case $\alpha=\gamma=0$ does not occur (does not come from a Sklyanin algebra). One can see it directly using the above formulas for $\alpha$ and $\gamma$. However there is an indirect way. The potential algebra with the potential $P_{0,0}$ enjoys a rank 1 quadratic relation $xy=yy$ (we mean the usual rank in $V\otimes V$). On the other hand, it is elementary to see that no such thing exists for $Q^{p,q,r}$ provided there are at least two non-zero numbers among $p$, $q$ and $r$ and $p^3=q^3=r^3$ fails. We shall discuss the potential $P_{0,0}$ later since it is peculiar indeed.

Thus we have that $A$ is isomorphic to the potential algebra $B$ difined by the potential
$P_{\alpha,\gamma}=x^3-xyz^\rcirclearrowleft+\alpha yyz^\rcirclearrowleft-(\gamma-\alpha^2)z^3$ with $\alpha,\gamma\in\K$. We know that $(\alpha,\gamma)\neq(0,0)$. Since $A$ and $B$ are isomorphic, the proof will be complete if we show that
Then $B$ is Koszul and $H_B=(1-t)^{-3}$. A direct computation shows that $B$ is presented by generators $x,y,z$ and the relations
\begin{equation}\label{dereB}
xx-zx+zy+\alpha zz=0,\ \ xy-yy-\alpha zx+\gamma zz=0,\ \ yz-zx+zy+\alpha zz=0
\end{equation}
(observe the leading monomials $xx$, $xy$ and $yz$). Let $I$ be the right ideal in $B$ generated by $y$ and $z$: $I=yB+zB$. For $f,g\in B$, we write $f=g\,({\rm mod}\,I)$ if $f-g\in I$. For each $k\in\Z_+$ consider the following property:
\begin{itemize}
\item[($\Pi_k$)] there exist $a_k,b_k\in \K$ such that $xz^kx=a_kxz^{k+1}\,({\rm mod}\,I)$ and $xz^ky=b_kxz^{k+1}\,({\rm mod}\,I)$.
\end{itemize}
Note that according to (\ref{dereB}), $\Pi_0$ is satisfied with $a_0=b_0=0$. We shall prove the following Claim.

{\bf Claim:} If $k\in\Z_+$ and $\Pi_k$ is satisfied, then either $\Pi_{k+1}$ is satisfied or
\begin{itemize}
\item[($\Sigma_{k+1}$)] $xz^{k+1}x=xz^{k+1}y\,({\rm mod}\,I)$ and $xz^{k+2}=0\,({\rm mod}\,I)$.
\end{itemize}

\begin{proof}[Proof of Claim] Assume that $\Pi_k$ is satisfied. Then there are $a_k,b_k\in \K$ such that $xz^kx=a_kxz^{k+1}\,({\rm mod}\,I)$ and $xz^ky=b_kxz^{k+1}\,({\rm mod}\,I)$. Using (\ref{dereB}), we see that $xz^kyz=xz^{k+1}x-xz^{k+1}y-\alpha xz^{k+2}$. On the other hand, $xz^kyz=b_kxz^{k+2}\,({\rm mod}\,I)$. These equalities yield
$$
-xz^{k+1}x+xz^{k+1}y+(b_k+\alpha)xz^{k+2}=0\,({\rm mod}\,I).
$$
By (\ref{dereB}), $xz^kxy=xz^kyy+\alpha xz^{k+1}x-\gamma xz^{k+2}$. Using ($\Pi_k$), we then have $xz^kxy=b_kxz^{k+1}y+\alpha xz^{k+1}x-\gamma xz^{k+2}\,({\rm mod}\,I)$. Directly from ($\Pi_k$), we have $xz^kxy=a_kxz^{k+1}y\,({\rm mod}\,I)$. These two equalities yield
$$
\alpha xz^{k+1}x+(b_k-a_k)xz^{k+1}y-\gamma xz^{k+2}=0\,({\rm mod}\,I).
$$
Again, by (\ref{dereB}), $xz^kxx=xz^{k+1}x-xz^{k+1}y-\alpha xz^{k+2}$. Using ($\Pi_k$), we have $xz^kxx=a_kxz^{k+1}x\,({\rm mod}\,I)$. These two equalities yield
$$
(a_k-1)xz^{k+1}x+xz^{k+1}y+\alpha xz^{k+2}=0\,({\rm mod}\,I).
$$
In the matrix form, the equations in the above three displays read
\begin{equation}\label{mat}
M\left(\begin{array}{c}xz^{k+1}x\\ xz^{k+1}y\\ xz^{k+2} \end{array}\right)=\left(\begin{array}{c}0\\ 0\\ 0\end{array}\right)\ ({\rm mod}\,I),
\ \ \text{where}\ \ M=\left(\begin{array}{ccc}-1&1&b_k+\alpha\\ \alpha&b_k-a_k&-\gamma\\ a_k-1&1&\alpha\end{array}\right).
\end{equation}
If the first two columns of matrix $M$ are linearly independent, $\Pi_{k+1}$ follows straight away.

It remains to consider the case when the first two columns of matrix $M$ are proportional. In this case $a_k=0$ and $b_k=-\alpha$. Thus in this case $M$ has the form
$$
M=\left(\begin{array}{ccc}-1&1&0\\ \alpha&-\alpha&-\gamma\\ -1&1&\alpha\end{array}\right).
$$
Since $(\alpha,\gamma)\neq (0,0)$, (\ref{mat}) is now equivalent to $xz^{k+1}x=xz^{k+1}y\,({\rm mod}\,I)$ and $xz^{k+2}=0\,({\rm mod}\,I)$, which is $\Sigma_{k+1}$. This concludes the proof of the claim.
\end{proof}

According to the above claim we have two options: either $\Pi_k$ holds for all $k\in\Z_+$ or for some $k\in\Z_+$, $\Pi_0,\dots,\Pi_k$ and $\Sigma_{k+1}$ are satisfied. Note that the first option is what happens for generic $(\alpha,\gamma)\in\K^2$, while the validity of $\Pi_0,\dots,\Pi_k$ and $\Sigma_{k+1}$ occurs for $(\alpha,\gamma)$ from a non-trivial algebraic variety.

{\bf Case 1:} \ $\Pi_k$ holds for all $k\in\Z_+$. The only monomials, which do not contain any of $yz$, $xz^kx$ and $xz^ky$ for $k\in\Z_+$ are $z^ky^m$ and $z^ky^mxz^j$ for $k,m,j\in\Z_+$. Denote this set of monomials $N$. The number of monomials of degree $n$ in $N$ is exactly $\frac{(n+1)(n+2)}{2}$. Since $\Pi_k$ is saisfied for each $k$, $B$ is the linear span of $N$. It follows that $\dim B_n\leq \frac{(n+1)(n+2)}{2}$ for each $n$ and these inequalities turn into equalities precisely when monomials from $N$ are linearly independent in $B$. On the other hand, by Lemma~\ref{loest}, $\dim A_n=\dim B_n\geq \frac{(n+1)(n+2)}{2}$ for each $n$. Hence $\dim B_n=\frac{(n+1)(n+2)}{2}$ for every $n$, that is, $H_B=(1-t)^{-3}$, and monomials from $N$ are linearly independent in $B$. Equalities in $\Pi_k$ can be written as the equalities $xz^kx-a_kxz^{k+1}+f_k$ and $xz^ky-b_kxz^{k+1}+g_k$ in $B$ with $f_k,g_k$ being homogeneous elements of $I$ of degree $k+2$. According to the above observations, the equality $H_B=(1-t)^{-3}$ implies that $yz-zx+zy+\alpha zz$ together with $xz^kx-a_kxz^{k+1}+f_k$ and $xz^ky-b_kxz^{k+1}+g_k$ for $k\in\Z_+$ form a reduced Gr\"obner basis in the ideal of relations of $B$ and that $N$ is the set of corresponding normal words. Since $N$ is closed under multiplication by $z$ on the left, the map $u\mapsto zu$ from $B$ to itself is injective. Thus $B$ has no non-trivial right annihilators. Since $A$ and $B$ are isomorphic $H_A=(1-t)^{-3}$  and $A$ has non non-trivial right annihilators. By Lemma~\ref{duser} $H_{A^!}=(1+t)^3$ and $wA_2^!\neq \{0\}$ for every non-zero $w\in A_1^!$. Now Lemma~\ref{a3} implies that $A$ is Koszul.

{\bf Case 2:} \ there is $k\in\Z_+$ such that $\Pi_0,\dots,\Pi_k$ and $\Sigma_{k+1}$ hold. By $\Sigma_{k+1}$, $xz^{k+1}x=xz^{k+1}y\,({\rm mod}\,I)$ and $xz^{k+2}=0\,({\rm mod}\,I)$. Using these equalities and (\ref{dereB}), we have $xz^{k+1}xx=xz^{k+1}yx\,({\rm mod}\,I)$ and $xz^{k+1}xx=xz^{k+2}x-xz^{k+2}y-\alpha xz^{k+3}=0\,({\rm mod}\,I)$. Hence
$$
xz^{k+1}yx=0\,({\rm mod}\,I).
$$
The only monomials, which do not contain any of $yz$, $xz^jx$ for $0\leq j\leq k+1$, $xz^jy$ for $0\leq j\leq k$, $xz^{k+2}$ and $xz^{k+1}yx$ are the words of the form $z^jy^mw$, where $m,j\in\Z_+$ and $w$ is an initial subword (empty allowed) of the infinite word $xz^{k+1}yyyy\dots$. Denote this set of monomials $N$. As before, the number of monomials of degree $n$ in $N$ is exactly $\frac{(n+1)(n+2)}{2}$. Now exactly the same argument as in Case~1 yields $H_A=H_B=(1-t)^{-3}$ and shows that $A$ is Koszul. Note that this time the Gr\"obner basis in the ideal of relations of $B$ turns out to be finite: it consisits of $2k+6$ elements. This concludes the proof of Theorem~\ref{copo0}.

\section{Isomorphic Sklyanin algebras}

The question when two different Sklyanin algebras are isomorphic as graded algebras is touched in \cite{ATV2}. The authors associate to each Sklyanin algebra some geometric data (an elliptic curve and its automorphism) and observe that Sklyanin algeras are isomorphic precisely, when the same happens to geometric data associated to them. This certainly may be useful on many occasions, however this characterization of isomorphic Sklyanin algebras is hard to use when the only data we have are just the triples of parameters (possible, of course, but rather inconvenient). In this section we find out explicitly in terms of just the values of parameters which Sklyanin algebras are isomorphic to each other.

One may wonder whether the answer depends on whether we allow all algebra isomorphisms instead of just graded (=linear in our case) ones. Well, it does not according to the following very general observation. It must be somewhere in the literature, however we could not come up with a reference. For this reason we include its proof. We banish the said proof to the next section devoted to remarks in order to keep the flow of arguments intact.

\begin{proposition}\label{grq} Let $A=A(V_1,R_1)$ and $B=A(V_2,R_2)$ be two quadratic algebras over the same ground field $\K$ $($any field of any characteristic is allowed here$)$. Assume also that $A$ and $B$ are isomorphic as algebras. Then they are also isomorphic as graded algebras. In other words, there is a linear change of variables turning $A$ into $B$.
\end{proposition}

Keeping the above proposition in mind, we shall just deal with isomorphisms in the category of graded algebras. Recall that we still assume that $\rm char\,\K\neq 3$. Throughout this section we shall also assume that there is $\theta\in\K$ such that $\theta^3=1\neq\theta$ and use $\theta$ for this element without further reference. Note that the absence of a non-trivial cubic root of $1$ does effect the results below both directly (when $\theta$ features in a statement) and indirectly.

First, we remove the degenerate Sklyanin algebras from the picture. We say that a Sklyanin algebra $Q^{p,q,r}$ is {\it degenerate} if $(pq,pr,qr)=(0,0,0)$ (there are at least two zeros among $p$, $q$ and $r$) or $p^3=q^3=r^3$. Note that according to Theorem~\ref{copo0}, degenerate Sklyanin algebras are precisely the Sklyanin algebras whose Hilbert series differs from $(1-t)^{-3}$.

\begin{lemma}\label{asas} Let $A=Q^{p,q,r}$ be a degenerate Sklyanin algebra. Then
\begin{itemize}\itemsep=-2pt
\item $A=\K\langle x,y,z\rangle\iff p=q=r=0;$
\item $A$ is isomorphic to the monomial algebra $Q^{1,0,0}$ given by the relations $xy=zx=yz=0$ if and only if either $r=0$, $pq=0$ and $(p,q)\neq (0,0)$ or $p^3=q^3=r^3\neq 0$ and $p\neq q;$
\item $A$ is to the monomial algebra $Q^{0,0,1}$ given by the relations $xx=yy=zz=0$ if and only if either $p=q=0$ and $r\neq 0$ or $p=q\neq 0$ and $p^3=r^3;$
\end{itemize}
Furthermore, the algebras $\K\langle x,y,z\rangle$, $Q^{1,0,0}$ and $Q^{0,0,1}$ are pairwise non-isomorphic and a degenerate Sklyanin algebra can not be isomorphic to a non-degenerate one.
\end{lemma}

\begin{proof} The statement $A=\K\langle x,y,z\rangle\iff p=q=r=0$ is obvious. If exactly two of $p$, $q$ and $r$ are zero, then up to scaling we have three options for $(p,q,r)$: $(1,0,0)$, $(0,1,0)$ and $(0,0,1)$. Note that swapping $x$ and $y$, while leaving $z$ as is, provides an isomorphism betweem $Q^{1,0,0}$ and $Q^{0,1,0}$. It remains to deal with the case $p^3=q^3=r^3\neq 0$. If $p=q$, then by scaling we can make $p=q=1$. Then $r^3=1$. By Lemma~\ref{root1}, $Q^{p,q,r}$ is isomorphic to $Q^{1,1,1}$. By Lemma~\ref{root2}, $Q^{1,1,1}$ is isomorphic to $Q^{p',q',r'}$, where $p'=\theta^2+\theta+1=0$, $q'=\theta+\theta^2+1=0$ and $r'=1+1+1=3\neq 0$. Hence $Q^{p,q,r}$ is isomorphic to $Q^{0,0,1}$. It remains to deal with the situation $p^3=q^3=r^3$ and $p\neq r$. By scaling, we can make $p=1$. Then $r^3=q^3=1$ and $q\neq 1$. By Lemma~\ref{root1}, $Q^{p,q,r}$ is isomorphic to $Q^{1,q,q^2}$. By Lemma~\ref{root2},  $Q^{1,q,q^2}$ is isomorphic to $Q^{p',q',r'}$, where $p'=q+q^3+q^2=0$, $q'=q^2+q^2+q^2=3q^2\neq 0$ and $r'=1+q+q^2=0$. Thus $Q^{p,q,r}$ is isomorphic to $Q^{1,0,0}$.

Clearly, $\K\langle x,y,z\rangle$ is not isomorphic to any of $Q^{1,0,0}$ and $Q^{0,0,1}$ (for example, $\K\langle x,y,z\rangle$ has no non-trivial zero divisors, while $Q^{1,0,0}$ and $Q^{0,0,1}$ have non-trivial zero divisors provided by the defining relations). The algebras $Q^{1,0,0}$ and $Q^{0,0,1}$ are non-isomorphic since the only one-dimensional representation of $Q^{0,0,1}$ is the augmentation map, while the map sending $x$ to $1$ and $y$ and $z$ to $0$ extends to a one-dimensional representation of $Q^{1,0,0}$ ($Q^{0,0,1}$ has only one one-dimensional representation, while $Q^{1,0,0}$ has more than one; as a matter of fact it has infinitely many of those if $\K$ is infinite).

Finally, by Theorem~\ref{copo0} and Lemma~\ref{dege1}, every degenerate Sklyanin algebra has exponential growth, while every non-degenerate Sklyanin algebra has polynomial growth. Hence a degenerate Sklyanin algebra can not be $\cal A$-isomorphic to a non-degenerate one.
\end{proof}

Note that monomiality of degenerate Sklyanin algebras was first noticed by Smith \cite{smi}. Lemma~\ref{asas} effectively takes degenerate algebras out of the picture, so we can concentrate on the non-degenerate ones. That is we have to deal with $Q^{p,q,r}$ for
$$
(p,q,r)\in M_0=\{(p,q,r)\in\K^3:(pq,pr,qr)\neq(0,0,0)\ \ \text{and}\ \ (p_3-q^3,p^3-r^3,q^3-r^3)\neq(0,0,0)\}.
$$
We split $M_0$ into two disjoint subsets $M_0=M_1\cup M_2$, where
$$
M_1=\{(p,q,r)\in\K^3:r\neq 0,\ (p,q)\neq (0,0),\ (p+q)^3+r^3\neq 0\ \ \text{and}\ \ (p_3-q^3,p^3-r^3,q^3-r^3)\neq(0,0,0)\}
$$
and $M_2=M_0\setminus M_1$. Just as we have removed degenerates, we shall do the same with $M_2$, which will leave us free to concentrate on $M_2$ afterwards.

\begin{lemma}\label{qp1} Let $(p,q,r)\in\K^3$ be such that $r\neq 0$, $(p_3-q^3,p^3-r^3,q^3-r^3)\neq(0,0,0)$ and $(p+q)^3+r^3=0$. Then $Q^{p,q,r}$ is isomorphic to the quantum polynomial algebra $Q^{1,-\alpha,0}$ given by the relations $xy=\alpha yx$, $zx=\alpha xz$ and $yz=\alpha zy$, where $\alpha=\theta\frac{p-\theta^2q}{p-\theta q}\in\K^*$.
\end{lemma}

\begin{proof} Clearly, $Q^{p,q,r}=Q^{s,t,1}$, where $s=\frac{p}{r}$ and $t=\frac{q}{r}$. Since $(p+q)^3+r^3=0$, we have $(-s-t)^3=1$. By Lemma~\ref{root1}, $Q^{s,t,1}$ is isomorphic to $Q^{s,t,-s-t}$. By Lemma~\ref{root2}, $Q^{s,t,-s-t}$ is isomorphic to $Q^{p',q',r'}$, where $p'=(\theta^2-1)s+(\theta-1)t$, $q'=(\theta-1)s+(\theta^2-1)t$ and $r'=0$. Note that the assumptions yield that $Q^{p,q,r}$ is non-degenerate and therefore so is $Q^{p',q',r'}$. Hence $p'q'\neq 0$. Then $Q^{p',q',r'}=Q^{1,-\alpha,0}$ with $\alpha=-\frac{(\theta-1)s+(\theta^2-1)t}{(\theta^2-1)s+(\theta-1)t}$. After easy simplifications, one gets $\alpha=\theta\frac{p-\theta^2q}{p-\theta q}\in\K^*$.
\end{proof}

\begin{lemma}\label{1di} For every $(p,q,r)\in M_2$, the only one-dimensional representation of $Q^{p,q,r}$ is the augmentation map.
\end{lemma}

\begin{proof} Let $A=Q^{p,q,r}$ with $(p,q,r)\in M_1$. We shall verify that $A$ has no one-dimensional representations other than the augmentation map. Let $\phi:A\to\K$ be a representaion (=an algebra homomorphism) and let $a=\phi(x)$, $b=\phi(y)$ and $c=\phi(z)$. We have to prove that $a=b=c=0$. From the defining relations of $A$ it follows that $(p+q)ab=-rc^2$, $(p+q)bc=-ra^2$ and $(p+q)ac=-rb^2$. Multiplying these equalities, we get $a^2b^2c^2((p+q)^3+r^3)=0$. Since  $(p,q,r)\in M_1$, $(p+q)^3+r^3\neq 0$. Hence $abc=0$. Without loss of generality, we may assume that $a=0$. From  $(p+q)ab=-rc^2$ and $(p+q)ac=-rb^2$ we now have $rb=rc=0$. Since $(p,q,r)\in M_1$, $r\neq 0$ and therefore $a=b=c=0$.
\end{proof}

\begin{lemma}\label{qupo} For every $(p,q,r)\in M_2$, $Q^{p,q,r}$ is isomorphic to  a
quantum polynomial algebra $B^\alpha=Q^{1,-\alpha,0}$ given by the relations $xy=\alpha yx$, $zx=\alpha xz$ and $yz=\alpha zy$ for some $\alpha\in\K^*$. Next, $B^\alpha$ and $B^\beta$ are isomorphic if and only if either $\alpha=\beta$ of $\alpha\beta=1$. Furthermore, $Q^{p,q,r}$ and $Q^{p',q',r'}$ are non-isomorphic if $(p,q,r)\in M_2$ and $(p',q',r')\in M_1$.
\end{lemma}

\begin{proof} If $(p,q,r)\in M_2$, from definition of $M_2$ it follows that either $r=0$ and $pq\neq 0$ or $r\neq 0$, $(p_3-q^3,p^3-r^3,q^3-r^3)\neq(0,0,0)$ and $(p+q)^3+r^3=0$. If $r=0$ and $pq\neq 0$, then $Q^{p,q,r}$ is isomorphic to $B^\alpha$ with $\alpha=-\frac{q}{p}\in\K^*$. Otherwise, Lemma~\ref{qp1} yields an isomorphism of $Q^{p,q,r}$ and $B^\alpha$ with $\alpha=\theta\frac{p-\theta^2q}{p-\theta q}\in\K^*$. In any case, $Q^{p,q,r}$ is isomorphic to  $B^\alpha$ with $\alpha\in\K^*$.

If $(p',q',r')\in M_1$, by Lemma~\ref{1di}, $Q^{p',q',r'}$ has just one one-dimensional representation, the augmentation map. On the other hand each $B^\alpha$ has plenty one-dimensional representations. For instance, the map $x\mapsto 1$, $y\mapsto 0$ and $z\mapsto 0$ extends to a one-dimensional representation for each $B^\alpha$. Since $Q^{p,q,r}$ for $(p,q,r)\in M_2$ is isomorphic to a $B^\alpha$, we see that $Q^{p,q,r}$ and $Q^{p',q',r'}$ are non-isomorphic if $(p,q,r)\in M_2$ and $(p',q',r')\in M_1$.

It remains to prove that $B^\alpha$ and $B^\beta$ are isomorphic if and only if either $\alpha=\beta$ of $\alpha\beta=1$. Since $B^1=\K[x,y,z]$ is the only commutative algebra among $B^\alpha$, it is non-isomorphic to any $B^\alpha$ with $\alpha\neq 1$. Thus it remains to deal with the case $\alpha\neq 1$ and $\beta\neq 1$. First, observe that swapping $x$ and $y$, while leaving $z$ as is, provides a $\cal G$-isomorphism between $B^\alpha$ and $B^{\alpha^{-1}}$. Now let $\alpha,\beta\in\K^*$, $\alpha\neq 1$, $\beta\neq 1$, $\alpha\neq \beta$ and $\alpha\beta\neq1$. The proof will be complete if we show that $B^\alpha$ and $B^\beta$ are non-isomorphic. Assume the contrary: there is a graded algebra isomorphism $\phi:B^\alpha\to B^\beta$. Let $x$, $y$, $z$ be the usual generators of $B^\beta$, while $u$, $v$ and $w$  be the images under $\phi$ of the usual generators of $B^\alpha$. Then $u$, $v$ and $w$ are homogeneous degree one elements of $B^\beta$, they form a linear basis in the degree $1$ component $B_1^\beta$ of $B^\beta$ and they satisfy $uv=\alpha vu$, $wu=\alpha uw$ and $vw=\alpha wv$. Furthermore, the linear span of $uv-\alpha vu$, $wu-\alpha uw$ and $vw-\alpha wv$ treated as elements of $\K\langle x,y,z\rangle$ must coincide with the linear span of $xy-\beta yx$, $zx-\beta xz$ and $yz-\beta zy$. In particular, $uv-\alpha vu$, $wu-\alpha uw$ and $vw-\alpha wv$ must be 'square-free': when written in terms of $x$, $y$ and $z$ they should not contain any of $xx$, $yy$ and $zz$. Since $\alpha\neq 1$, it immediately follows that each of $u$, $v$ and $w$, when written in terms of $x$, $y$ and $z$ must be a scalar multiple of a single element of $\{x,y,z\}$. That is, the matrix of our linear substitution $\phi$ is a product of a diagonal matrix and a permutation matrix (of size $3\times 3$). If the permutation in question is even, one easily sees that $uv-\alpha vu$, $wu-\alpha uw$ and $vw-\alpha wv$ are scalar multiples of $xy-\alpha yx$, $zx-\alpha xz$ and $yz-\alpha zy$ (in some order), while if the permutation is odd  $uv-\alpha vu$, $wu-\alpha uw$ and $vw-\alpha wv$ are scalar multiples of $\alpha xy-yx$, $\alpha zx- xz$ and $\alpha yz-\alpha zy$ (in some order). Thus, the spans of the triples $uv-\alpha vu$, $wu-\alpha uw$ and $vw-\alpha wv$ and
$xy-\beta yx$, $zx-\beta xz$ and $yz-\beta zy$ can coincide only if $\alpha=\beta$ or $\alpha\beta=1$. This contradiction completes the proof.
\end{proof}

Next lemma clinches description of isomorphic Sklyanin algebras. Since every $(p,q,r)\in M_1$ satisfies $r\neq 0$, by dividing by $r$, $(p,q,r)$ can be turned into a unique triple $(a,b,1)$ with
$$
(a,b)\in M,\ \ \text{where}\ \ M=\{(a,b)\in\K^2:(a,b)\neq (0,0),\ (a+b)^3+1\neq 0,\ (a^3-1,b^3-1)\neq (0,0)\}.
$$
Since scaling the triple of parameters does not change the Sklyanin algebra, in order to describe which Sklyanins with $(p,q,r)\in M_1$ are isomorphic, it suffices to do so in the case $(p,q,r)=(a,b,1)$ with $(a,b)\in M$.

\begin{lemma}\label{m1fi} If both $(a,b)$ and $(a',b')$ belong to $M$, then the Sklyanin algebras $Q^{a,b,1}$ and $Q^{a',b',1}$ are isomorphic if and only if $(a,b)$ and $(a',b')$ are in the same orbit of the group action on $M$ generated by two maps $(a,b)\mapsto (\theta a,\theta b)$ and $(a,b)\mapsto \left(\frac{\theta a+\theta^2b+1}{a+b+1},\frac{\theta^2 a+\theta b+1}{a+b+1}\right)$. This group is finite, consists of $24$ elements $($thus if $\K$ is infinite, for generic $(a,b)\in M$, there are exactly $23$ other elements of $M$ giving rise to an isomorphic Sklyanin algebra$)$ and is isomorphic to $SL_2(\Z_3)$.

The complete list of pairs $(a',b')\in M$ such that for a given $(a,b)\in M$, $Q^{a,b,1}$ and $Q^{a',b',1}$ are isomorphic is as follows$:$
\begin{itemize}\itemsep=-2pt
\item $(\theta^ja,\theta^jb)$ with $j\in\{0,1,2\};$
\item $(\theta^jb,\theta^ja)$ with $j\in\{0,1,2\};$
\item $\left(\frac{\theta^ja+\theta^kb+\theta^m}{a+b+\theta^n},\frac{\theta^ka+\theta^jb+\theta^m}{a+b+\theta^n}\right)$ with $n\in\{0,1,2\}$ and $\{j,k,m\}=\{0,1,2\}.$
\end{itemize}
The last condition means that $(j,k,m)$ is a permutation of $(0,1,2)$. Each of the first two lines in the above list provide $3$ pairs, while the last line yields $18$.
\end{lemma}

\begin{proof} Assume that $(a,b)$ and $(a',b')$ belong to $M$, then the Sklyanin algebras $A=Q^{a,b,1}=A(V,R)$ and $B=Q^{a',b',1}=A(V,R')$ are isomorphic. By Theorem~\ref{copo0}, $\dim A_3=\dim B_3=10$. Since $\dim V^3=27$, we have $\dim(VR+RV)=\dim(VR'+R'V)=27-10=17$. Since $\dim VR=\dim RV=\dim VR'=\dim R'V=9$, it follows that $\dim(RV\cap VR)=\dim(R'V\cap V'R)=1$. On the other hand, obviously, the potentials $F=x^3+y^3+z^3+axyz^\rcirclearrowleft+bxzy^\rcirclearrowleft$ and $F'=x^3+y^3+z^3+a'xyz^\rcirclearrowleft+b'xzy^\rcirclearrowleft$ for $A$ and $B$ respectively satisfy $F\in RV\cap VR$ and $F'\in R'V\cap VR'$. Since a linear substitution $T\in GL_3(\K)$ facilitating the graded isomorphism of $A$ and $B$ must send $RV\cap VR$ to $R'V\cap VR'$, it transforms $F$ to $F'$ up to a scalar multiple. Hence $T$ must transform the abelianization of $F$ $G=x^3+y^3+z^3+3(a+b)xyz\in\K[x,y,z]$ (the image of $F$ under the canonical map from $\K\langle x,y,z\rangle$ to $\K[x,y,z]$) into the abelianization $G'=x^3+y^3+z^3+3(a'+b')xyz\in\K[x,y,z]$ of $F'$ up to a scalar multiple. Hence, $T$ provides an isomorphism between the elliptic (projective) curves $C$ given by $G=0$ and $C'$ given by $G'=0$.

Since $(a,b)$ and $(a',b')$ belong to $M$, we have $(a+b)^3+1\neq 0$ and $(a'+b')^3+1\neq 0$ and therefore the curves $C$ and $C'$ are regular. Note that the each of the curves $C$ and $C'$ have the same exactly collection of nine inflection points which are the nine lines $L_k$ for $1\leq k\leq 9$ spanned by $(1,-1,0)$, $(1,-\theta,0)$, $(1,-\theta^2,0)$, $(1,0,-1)$, $(1,0,-\theta)$, $(1,0,-\theta^2)$, $(0,1,-1)$, $(0,1,-\theta)$ and $(0,1,-\theta^2)$ respectively. Since $T$ is an isomorphism between $C^{p,q}$ and $C^{p',q'}$, $T$ must leave the union of $L_j$ invariant. It is a routine exercise to verify that the subgroup $G$ of $GL_3(\K)$ leaving the union of $L_j$ invariant is generated by
\begin{equation}\label{GK}
\left(\begin{array}{ccc}\lambda&0&0\\ 0&\lambda&0\\ 0&0&\lambda\end{array}\right)\ \ (\lambda\in\K^*),\ \
\left(\begin{array}{ccc}0&1&0\\ 0&0&1\\ 1&0&0\end{array}\right),\ \
\left(\begin{array}{ccc}1&0&0\\ 0&1&0\\ 0&0&\theta\end{array}\right)\ \ \text{and}\ \  \left(\begin{array}{ccc}\theta&\theta^2&1\\ \theta^2&\theta&1\\ 1&1&1\end{array}\right).
\end{equation}
One way to see it is to notice that, first of all, each linear map in the above display leaves the union of $L_j$ invariant. Next, one easily checks that the group $K$ generated by matrices in the above display acts transitively on the set of triples of $L_j$ that span $\K^3$. This fact together with an (easily verifiable) observation that $K$ contains all permutation matrices and all scalar matrices yields that $K$ contains all elements of $G$ and therefore $K=G$. Thus $T\in G$.

Now the linear transformations given by the first two matrices in the above display provide automorphisms of each $Q^{p,q,1}$, while the linear transformation given by the third matrix in the above display facilitates an isomorphism between $Q^{p,q,1}$ and $Q^{\theta p,\theta q,1}$ for each $(p,q)\in M$. Finally, a direct computation shows that the linear transformations given by the last matrix in the above display provides an isomorphism of $A^{p,q}$ and $A^{p',q'}$ for every $(p,q)\in M$, where
$(p',q')=\bigl(\frac{\theta p+\theta^2q+1}{p+q+1},\frac{\theta^2p+\theta q+1}{p+q+1}\bigr)$. This completes the proof of the isomorphism statement. Thus the existence of an isomorphism between $A$ and $B$ is equaivalent to $A$ and $B$ being in the same orbit of the group action on $M$ generated by two maps $(a,b)\mapsto (\theta a,\theta b)$ and $(a,b)\mapsto \left(\frac{\theta a+\theta^2b+1}{a+b+1},\frac{\theta^2 a+\theta b+1}{a+b+1}\right)$.

A direct computation shows that this group consists of the maps $(a,b)\mapsto (\theta^ja,\theta^jb)$ with $j\in\{0,1,2\}$
$(a,b)\mapsto (\theta^jb,\theta^ja)$ with $j\in\{0,1,2\}$ and $(a,b)\mapsto \left(\frac{\theta^ja+\theta^kb+\theta^m}{a+b+\theta^n},\frac{\theta^ka+\theta^jb+\theta^m}{a+b+\theta^n}\right)$ with $n\in\{0,1,2\}$ and $\{j,k,m\}=\{0,1,2\}$. Hence the group has $24$ elements. As for this group being isomorphic to $SL_2(\Z_3)$, this can be done by computing enough features of this group (for instance, it has a two-element centre, maximal order of an element in it is $6$ etc.) to be able to identify it in the well-known list of $24$-element groups.
\end{proof}

\section{Appendix}

We start this section by proving Proposition~\ref{grq}.

\subsection{Proof of Proposition~\ref{grq}}

Let $A=A(V_1,R_1)$ and $B=A(V_2,R_2)$ be quadratic algebras over the same ground field $\K$. Let $x_1,\dots,x_n$ be a linear basis in $V_1$ and $y_1,\dots,y_m$ be a linear basis in $V_2$. Assume also that $f_1,\dots,f_s$ is a linear basis in $R_1$, while $g_1,\dots,g_t$ is a linear basis in $R_2$. Assume also that there is an algebra isomorphism $\phi:A\to B$. The proof will be complete if we construct another algebra isomorphism $\psi:A\to B$ such that its restriction to $V_1$ is a linear isomorphism from $V_1$ onto $V_2$ (equivalently, $\psi$ is a graded algebra isomorphism). Since $f_j\in V_1^2$, $f_j=\sum\limits_{1\leq k,p\leq n}a^{(j)}_{k,p}x_kx_p$ with $a^{(j)}_{k,p}\in \K$. Let $u^{(j)}=\phi(x_j)$ for $1\leq j\leq n$. As usual, for an element $u$ of a graded algebra, $u_j$ stands for degree $j$ homogeneous component of $u$. Since $\phi$ is an algebra isomomorphism, $u^{(j)}$ must generate $B$. Since the first component of every element of $B$ in the subalgebra generated by $u^{(j)}$ is in the linear span of $u^{(j)}_1$, the said span must coincide with $V_2$. In particular, $n\geq m$. The same argument with the role of $A$ and $B$ switched yields $m\geq n$. Hence $m=n$ and $u^{(j)}_1$ for $1\leq j\leq n$ form a linear basis in $V_2$. Since $\phi$ is an algebra homomorphism, we have $\phi(f_j)=0$ in $B$, that is,
\begin{equation}\label{e11}
\sum\limits_{1\leq k,p\leq n}a^{(j)}_{k,p}u^{(k)}u^{(p)}=0\ \ \text{in $B$ for $1\leq j\leq s$.}
\end{equation}
The zero component of the above equation reads
\begin{equation}\label{e110}
\sum\limits_{1\leq k,p\leq n}a^{(j)}_{k,p}u_0^{(k)}u_0^{(p)}=0\ \ \text{in $\K$ for $1\leq j\leq s$.}
\end{equation}
The degree $1$ component of (\ref{e11}) looks like
\begin{equation}\label{e111}
\sum\limits_{1\leq k,p\leq n}a^{(j)}_{k,p}(u_0^{(k)}u_1^{(p)}+u_0^{(p)}u_1^{(k)})=0\ \ \text{in $V_2$ for $1\leq j\leq s$.}
\end{equation}
Since $u^{(j)}_1$ are linearly independent, the left-hand sides in (\ref{e111}) must be zero as polynomials in $u^{(j)}_1$. Hence,
\begin{equation}\label{e11i}
\sum\limits_{1\leq k,p\leq n}a^{(j)}_{k,p}(u_0^{(k)}u^{(p)}+u_0^{(p)}u^{(k)})=0\ \ \text{in $B$ for $1\leq j\leq s$.}
\end{equation}
Denote $v^{(j)}=u^{(j)}-u^{(j)}_0$. Combining (\ref{e11}), (\ref{e110}) and (\ref{e11i}) we get
\begin{equation}\label{e11v}
\sum\limits_{1\leq k,p\leq n}a^{(j)}_{k,p}(u^{(k)}-u^{(k)}_0)(u^{(p)}-u^{(p)}_0)=0\ \ \text{in $B$ for $1\leq j\leq s$.}
\end{equation}
The second degree component of (\ref{e11v}) now is
\begin{equation}\label{e11v1}
\sum\limits_{1\leq k,p\leq n}a^{(j)}_{k,p}u^{(k)}_1u^{(p)}_1=0\ \ \text{in $B$ for $1\leq j\leq s$.}
\end{equation}
Since $f_j$ are linearly independent and $u^{(j)}_1$ are linearly independent, (\ref{e11v1}) provides $s$ linearly independent quadratic relations in $B$. In particular, $s\geq t$. The same argument with the roles of $A$ and $B$ reversed, gives $t\geq s$. Hence $s=t$ and the linear span of the left-hand sides of (\ref{e11v}) must coincide with $R_2$. It follows that the map $x_i\mapsto u^{(i)}_1$ extends to an algebra isomorphism of $A$ and $B$. Clearly, this is the graded isomorphism, we were after.

\subsection{Some remarks}

Observe that our results on Sklyanin algebras collapse if the characteristic of the ground field equals $3$. Indeed, a cubic root of 1 plays an essential role in the substitution we construct as well as in the description of isomorphic Sklyanin algebras.

As for the potential
$$
P_{0,0}=x^3-xyz^\rcirclearrowleft,
$$
which does not correspond to a Sklyanin algebra, the algebra $W$ it does generate is peculiar indeed. The defining relations of $W$ are $xx-zx+zy=0$, $xy-yy=0$ and $yz-zx+zy=0$. A direct computation shows that the reduced Gr\"obner basis in the ideal of relations of $W$ is finite. It comprises $xx-zx+zy$, $xy-yy$, $yz-zx+zy=0$, $yyy$, $xzx-xzy+zyx-zzx+zzy$ and $xzyx$, which allows us to compute the Hilbert series of $W:$ $H_W=\frac{(1+t)(1+t^2)(1+t+t^2)}{1-t-t^3-2t^4}$. On the other hand, the dual $W^!$ is given by the relations $xx+yz+zx$, $xy+yy$, $xz$, $yx$, $zx+zy$ and $zz$. Those together with $yyy$, $zyz$, $yyz-zyy$ and $yzy-zyy$ form the reduced Gr\"obner basis in the ideal of relations of $W^!$. The corresponding normal words are $1$, $x$, $y$, $z$, $yy$, $yz$, $zy$ and $zyy$, yielding $H_{W^!}=(1+t)^3$. Clearly, the the duality relation $H_W(t)H_{W^!}(-t)=1$ fails and therefore $W$ is non-Koszul. Thus $W$ provides an example of a non-Koszul quadratic potential algebra on three generators.

We did not bother to find out which Sklyanin algebras are isomorphic in two cases: when ${\rm char}\,\K=3$ and when ${\rm char}\,\K\neq 3$ but $\K$ possesses no nontrivial cubic roots of $1$.

\normalsize

\bigskip

{\bf Acknowledgements.} \ We are grateful to IHES and MPIM for hospitality, support, and excellent research atmosphere. This work is funded by the ERC grant 320974, the ESC Grant N9038, and EPSRC grant EP/T029455/1.

\small\rm

\normalsize

\vskip1truecm

\scshape

\noindent   Natalia Iyudu

\noindent  Department of Mathematics and Statistics

\noindent Lancaster University, Lancaster, LA1 4YF

\noindent E-mail address: n.joudu@yahoo.de; iyudu@mpim-bonn.mpg.de;

\noindent n.iyudu@ihes.fr; n.iyudu@lancaster.ac.uk.

\bigskip

\
\vskip1truecm

\noindent    Stanislav Shkarin

\noindent Queens's University Belfast

\noindent Department of Pure Mathematics

\noindent University road, Belfast, BT7 1NN, UK

\noindent E-mail addresses: \qquad  {\rm
and}\ \ \ {\tt s.shkarin@qub.ac.uk}

\end{document}